\documentclass[12pt,a4paper]{article}
\usepackage[utf8]{inputenc}
\usepackage{amsmath}
\usepackage{amscd}
\usepackage{amsthm}
\usepackage{ tipa }
\usepackage{ upgreek }
\usepackage{amssymb}
\usepackage{hyperref}
\usepackage{enumitem}
\usepackage{url}
\usepackage{times}
\newtheorem{theorem}{Theorem}
\newtheorem{theoremA}{Theorem}

\newtheorem{lemma}{Lemma}
\newtheorem{remark}{Remark}
\numberwithin{equation}{section}

\allowdisplaybreaks    

\newcommand\blfootnote[1]{%
  \begingroup
  \renewcommand\thefootnote{}\footnote{#1}%
  \addtocounter{footnote}{-1}%
  \endgroup
}
\title{On the distribution of the sequence of integers $d(n^2)$}
\author{K. VENKATASUBBAREDDY\\Email: \href{mailto:20mmpp02@uohyd.ac.in}{20mmpp02@uohyd.ac.in} \and A. SANKARANARAYANAN\\Email: \href{mailto:sank@uohyd.ac.in}{sank@uohyd.ac.in}}
\date{}

\begin{document}

\maketitle
\begin{abstract}
In this paper, we study the distribution of the sequence of integers $d(n^2)$ under the assumption of the strong Riemann hypothesis. Under this assumption, we provide a refined asymptotic formula for the sum $\displaystyle\sum_{n\leq x}d(n^2)$ with an improved error term by extracting some more main terms.
\end{abstract}
\blfootnote{2010 AMS \emph{Mathematics subject classification.} Primary 11M06, 11M26.}

\section{Introduction}
An important and interesting problem in Analytic Number Theory is to study and estimate the error term $E(x)$ in the asymptotic formula of the distribution of the divisor function $d(n)$,
\begin{equation*}
    \text{i. e., }\sum_{n\leq x}d(n)=x\log x+(2\gamma-1)x+E(x)
\end{equation*}
and is known as the Dirichlet divisor formula. Dirichlet was the first to study and estimate by elementary arguments that $E(x)\ll x^{\frac{1}{2}}$ (see \cite{Ivic2012}), and after that many researchers studied and improved this bound by various methods and techniques, and the best known bound is due to Huxley \cite{Huxley}, and is $E(x)\ll x^{\frac{131}{416}+\varepsilon}$.

The omega results for this error term state that the best possible bound that one can establish is:

By Hardy \cite{Ivic2012}, the omega result for the error term $E(x)$ is
 \begin{equation*}
    E(x)=\begin{cases}
        \Omega_+\left((x\log x)^\frac{1}{4}\log \log x\right)\\
        \Omega_-(x^\frac{1}{4}).
         \end{cases}
\end{equation*}
It is still an unproven conjecture that $E(x)=O\left(x^{\frac{1}{4}+\varepsilon}\right)$. The Hardy omega result was consequently improved by Gangadharan, Corr\'adi and K\'atai, Hafner (see \cite{Corradi and Katai}, \cite{Gangadharan}, \cite{Hafner}), and for the best known bound, see \cite{Hafner}.

Note that the generating $L$-function attached to the coefficients $d(n)$ is $\zeta^2(s)$, where $\zeta(s)$ is the Riemann zeta function defined as
\begin{equation*}
    \zeta(s)=\sum_{n=1}^\infty\frac{1}{n^s}
\end{equation*}
 for $\Re(s)>1$, which has meromorphic continuation to the whole of complex plane $\mathbb{C}$ by the functional equation
\begin{equation}
    \zeta(s)=\chi(s)\zeta(1-s)\label{E1.1}
\end{equation}
and the conversion factor $\chi(s)$ satisfies
\begin{equation}
    \mid\chi(s)\mid\sim |t|^{\frac{1}{2}-\sigma}\label{E1.2}
\end{equation} 
for $s=\sigma+it$,  $|t|\geq t_0>10$ (see \cite{Titchmarsh}).

\noindent
\textbf{Riemann hypothesis:} All the nontrivial zeros of the Riemann zeta function $\zeta(s)$ lie on the critical line $\Re(s) = \frac{1}{2}$. 

\noindent
\textbf{Stronger Riemann hypothesis:} All the nontrivial zeros of the Riemann zeta function $\zeta(s)$ lie on the critical line $\Re(s) = \frac{1}{2}$ and each such zero is simple. 

In this paper, we are interested in the $L$-functions $\displaystyle\frac{\zeta(s)}{\zeta(2s)},\ \frac{\zeta^2(s)}{\zeta(2s)},\ \frac{\zeta^3(s)}{\zeta(2s)}$ and $\displaystyle\frac{\zeta^4(s)}{\zeta(2s)}$ and the distribution of the corresponding arithmetical functions. In contrast, the corresponding arithmetical functions are $|\mu(n)|,\ 2^{\omega(n)},\ d(n^2)$ and $d(n)^2$, respectively, where $\mu(n)$ is the M\"obius function and $\omega(n)$ denotes the number of distinct prime divisors of $n$. In particular, we are interested in the distributions $\displaystyle\sum_{n\leq x}|\mu(n)|,\ \sum_{n\leq x}2^{\omega(n)},\ \sum_{n\leq x}d(n^2)$ and $\displaystyle\sum_{n\leq x}d(n)^2$. Investigations on these sums have already been carried out, and their precise results are as follows (see \cite{Ivic2012}):

We have
\begin{align}
    &\sum_{n\leq x}|\mu(n)|=\frac{x}{\zeta(2)}+O\left(x^\frac{1}{2}\exp \left(-C\varepsilon(x)\right)\right),\label{E1.3}\\
    &\sum_{n\leq x}2^{\omega(n)}=A_1x\log x+A_2x+O\left(x^\frac{1}{2}\exp \left(-C\varepsilon(x)\right)\right),\label{E1.4}\\
   & \sum_{n\leq x}d(n^2)=B_1x\log^2x+B_2x\log x+B_3x+O\left(x^\frac{1}{2}\exp \left(-C\varepsilon(x)\right)\right),\label{E1.5}\\
   & \sum_{n\leq x}d(n)^2=C_1x\log^3x+C_2\log^2x+C_3\log x+C_4x+O\left(x^{\frac{1}{2}+\varepsilon}\right),\label{E1.6}
\end{align}
where $\varepsilon(x)=(\log x)^\frac{3}{5}(\log \log x)^{-\frac{1}{5}}$ and $C>0,\ A_1,\hdots, C_4$ are constants, which can be evaluated explicitly. 

In \cite{Ivic2012}, it is also mentioned that under the assumption of the Riemann hypothesis, the error term in \eqref{E1.3} can be improved to $\displaystyle O\left(x^{\frac{2}{5}}\exp\left(\frac{C\log x}{\log \log x}\right)\right)$. In \cite{KVAS}, we have improved the error term appearing in $\eqref{E1.4}$ by extracting some more main terms from the error term, unconditionally and under the assumption of the strong Riemann hypothesis, and our precise results are as follows:

\begin{theoremA}
 Under the assumption of the strong Riemann hypothesis, that is, all the nontrivial zeros of $\zeta(s)$ and $\zeta(2s)$ lie on the respective critical lines, and each such zero is simple, we have 
  \begin{equation*}
      \sum_{n\leq x}2^{\omega(n)}=\mathcal{A}_1^{'}x\log x+\mathcal{A}_2^{'}x+\sum_{\substack{\zeta(2\rho)=0,\ \Re (\rho)=\frac{1}{4},\\
0<\left| \Im (2\rho)=\gamma_\frac{1}{4}\right|<\frac{x}{10}}}\mathcal{A}_{\gamma_\frac{1}{4}}^{'}x^{\frac{1}{4}+i\gamma_\frac{1}{4}}+O\left(x^{\varepsilon}\right),
  \end{equation*}
   where $\mathcal{A}_1^{'}=\frac{1}{\zeta(2)}$ and $\mathcal{A}_2^{'}=\frac{2\gamma-1}{\zeta(2)}$ and $\mathcal{A}_{\gamma_\frac{1}{4}}^{'}$ are some effective complex constants. 
\end{theoremA}

\begin{theoremA}
Unconditionally, we have
   \begin{equation*}
      \sum_{n\leq x}2^{\omega(n)}=\mathcal{A}_1x\log x+\mathcal{A}_2x+\sum_{\substack{\rho=\beta+i\gamma\\
      0\leq\beta\leq \frac{1}{2}\\0<|\gamma|\leq x^{\frac{21}{29}}}}\mathop{\mathrm{Res}}_{s=\rho} \frac{\zeta^2(s)}{\zeta(2s)}\frac{x^s}{s}+O\left(x^{\frac{8}{29}+\varepsilon}\right),
  \end{equation*}
  where $\mathcal{A}_1^{'}=\frac{1}{\zeta(2)}$ and $\mathcal{A}_2^{'}=\frac{2\gamma-1}{\zeta(2)}$.
\end{theoremA}

Now, this paper aims to study and establish a refined asymptotic formula for the sum $$\sum_{n\leq x}d(n^2)$$ by extracting some more main terms from the error term in \eqref{E1.5} under the assumption of the strong Riemann hypothesis.

Throughout the paper, $\varepsilon$ is a small positive constant, need not be the same at each occurrence. 

Precisely, we prove: 

\begin{theorem}
  Let $x>0$ be large and $\varepsilon>0$. Under the assumption of the strong Riemann hypothesis for $\zeta(s)$ and $\zeta(2s)$, that is, all the nontrivial zeros of $\zeta(s)$ and $\zeta(2s)$ lie on the respective critical lines, and each such zero is simple, we have 
 \begin{align*}
    \sum_{n\leq x}d(n^2)=&\mathcal{A}_1x(\log x)^2+\mathcal{A}_2 x\log x+\mathcal{A}_3x\\
    &+\sum_{\substack{\rho=\frac{1}{4}+i\gamma_\frac{1}{4}\\\zeta(2\rho)=0,\ |\gamma_\frac{1}{4}|<x^{\frac{2}{3}}}}\mathcal{A}_{\gamma_\frac{1}{4}}x^{\frac{1}{4}+i\gamma_\frac{1}{4}}+O\left(x^{\frac{1}{3}+\varepsilon}\right).
\end{align*}
   where $\mathcal{A}_1$, $\mathcal{A}_2$ and $\mathcal{A}_3$ are some real constants and $\mathcal{A}_{\gamma_\frac{1}{4}}$ are some complex constants which can be evaluated explicitly.
   \label{T1}
\end{theorem}
\begin{remark}
    It is not difficult to improve the above error term to $O\left( x^{\frac{1}{\theta}+\varepsilon}\right)$ for any large positive integer $\theta>3$ with $\frac{1}{\theta}<\frac{1}{3}$, by moving the line of integration to $\Re(s)=-\theta+2+\varepsilon$ in our proof. However, it should be noted that there exist some more main terms that come from the residues of $\displaystyle\frac{\zeta^3(s)}{\zeta(2s)}\frac{x^s}{s}$ at the poles $s=-1,\ -3,\ -5,\ \cdots, -\left(2\left[\frac{\theta}{2}\right]-3\right)$, and also the length of the summand on the right-hand side of the above asymptotic formula would be $x^{\frac{\theta-1}{\theta}}$, which is comparatively $x^{\frac{\theta-1}{\theta}}\sim x$ for large $\theta$.
\end{remark}
\begin{remark}
    From a general $\Omega$-result of Balasubramanian and Ramachandra \cite{RBKR}, we observe that if 
    \begin{equation*}
        \sum_{n\leq x}d(n^2)=M(x)+E_1(x)
    \end{equation*}
    then $E_1(x)=\Omega\left(x^{\frac{1}{3}-\varepsilon}\right)$. Thus, we are tempted to propose a conjecture here that the sum 
    \begin{equation*}
        \sum_{\substack{\rho=\frac{1}{4}+i\gamma_\frac{1}{4}\\\zeta(2\rho)=0,\ |\gamma_\frac{1}{4}|<x^{\frac{2}{3}}}}\mathcal{A}_{\gamma_\frac{1}{4}}x^{\frac{1}{4}+i\gamma_\frac{1}{4}}\ll x^{\frac{1}{3}+\varepsilon}
    \end{equation*}
    because of the oscillations of the factors $x^{i\gamma_\frac{1}{4}}$, though we are not in a position to substantiate this.
\end{remark}
\begin{remark}
   Although we have succeeded in obtaining refined error terms in equations \eqref{E1.3}, \eqref{E1.4}, and \eqref{E1.5} under the assumption of the strong Riemann hypothesis, it should be noted that it is very difficult to improve the error term in \eqref{E1.6} even under the assumption of the strong Riemann hypothesis following the arguments of this paper. This is due to the numerator factor $\zeta^4(s)$. When we apply the functional equation \eqref{E1.1} of the Riemann zeta function and the grown estimate \eqref{E1.2} for the conversion factor, the growth of $\displaystyle\frac{\zeta^4(s)}{\zeta(2s)}$ becomes $\displaystyle \ll |t|^{4(\frac{1}{2}-\sigma)-(\frac{1}{2}-2\sigma)}=|t|^{\frac{3}{2}-2\sigma}$ in certain half-planes (i.e. $\sigma<0)$, which is uncontrollably large and consequently makes the situation more difficult, and there is a strong resistance. In this case, we note that the best unconditional result to date is (see \cite{Jia and AS} and also \cite{KRAS})
   \begin{equation*}
       \sum_{n\leq x}d(n)^2=C_1x\log^3x+C_2\log^2x+C_3\log x+C_4x+O\left(x^{\frac{1}{2}}(\log x)^5\right).
   \end{equation*}
  It should be noted that the error term here is $\Omega(x^{\frac{3}{8}-\varepsilon})$ (see \cite{RBKR}), and still there is a long gap. 
\end{remark}

\section{Lemmas}
\begin{lemma}
    For $\Re(s)>1$, we have
\begin{equation*}
    F(s)=\sum_{n=1}^\infty\frac{d(n^2)}{n^s}=\frac{\zeta^3(s)}{\zeta(2s)}.
\end{equation*}\label{L1}
\end{lemma}
\begin{proof}
   We observe that $d(n^2)$ is multiplicative and $d(n^2)\ll_\varepsilon n^\varepsilon$, for any $\varepsilon>0$, and $|p^{-s}|<1$ for each prime $p$ and $\Re(s)>1$. Hence, the Euler product representation of $F(s)$ is given by
  \begin{align*}
    F(s)=&\prod_p(1+d(p^2)p^{-s}+d(p^4)p^{-2s}+d(p^6)p^{-3s}+d(p^8)p^{-4s}+\cdots)\\
    =&\prod_p(1+3p^{-s}+5p^{-2s}+7p^{-3s}+9p^{-4s}+\cdots)\\
    =&\prod_p\left(1-p^{-s}\right)^{-3}\left(1-p^{-2s}\right)\\
    =&\frac{\zeta^3(s)}{\zeta(2s)},
\end{align*}
which follows from the Euler product representation of the Riemann zeta function.
\end{proof}

\begin{lemma}
    Assuming the strong Riemann hypothesis, that is, that the nontrivial zeros of $\zeta(s)$ and $\zeta(2s)$ lie on the respective critical lines, we have the following residues of $F(s)\frac{x^s}{s}$ at the respective poles.
    \begin{align*}
   \mathop{\mathrm{Res}}_{s=1} F(s)\frac{x^s}{s}=&\mathcal{A}_1x(\log x)^2+\mathcal{A}_2 x\log x+\mathcal{A}_3x,   \\
   \mathop{\mathrm{Res}}_{s=0}\ F(s)\frac{x^s}{s}=&\zeta^2(0)\\
    \mathop{\mathrm{Res}}_{s=\frac{1}{4}+i\gamma_\frac{1}{4}}\ F(s)\frac{x^s}{s}=&\mathcal{A}_{\gamma_\frac{1}{4}}x^{\frac{1}{4}+i\gamma_\frac{1}{4}}    
   \end{align*}
   where $\frac{1}{4}+i\gamma_\frac{1}{4}$ are the coordinates of the simple zeros of $\zeta(2s)$ on the line $\Re(s)=\frac{1}{4}$ and the $\mathcal{A}_i$'s are real effective, and the $\mathcal{A}_{\gamma_i}$'s are complex effective constants.  \label{L2}
\end{lemma}

\begin{proof}
    It is well-known that $\zeta(2)=\frac{\pi^2}{6}\neq0$ and the Laurent series expansion of $\zeta(s)$ about $s=1$ is given by
   \begin{equation*}
        \zeta(s)=\frac{1}{s-1}+\sum_{m=0}^\infty\frac{(-1)^m}{m!}\gamma_m(s-1)^m,
    \end{equation*}
    where $\gamma_0,\gamma_1, \gamma_2,\cdots$ are defined as
    \begin{equation*}
       \gamma_m=\lim_{n\rightarrow \infty} \left\{\left(\sum_{k=1}^n\frac{(\log k)^m}{k}\right)-\frac{(\log n)^{m+1}}{m+1}\right\}.
    \end{equation*}
     When $m=0$, $\gamma_0=\gamma$ is the usual Euler-Mascheroni constant known as the Euler constant (see \cite{Ivic2012}).

     Thus, we can write 
     \begin{align*}
          &\mathop{\mathrm{Res}}_{s=1} F(s)\frac{x^s}{s}\\
          &=\mathop{\mathrm{Res}}_{s=1}\frac{\zeta^2(s)}{\zeta(2s)}\frac{x^s}{s}\\
          &=\mathop{\mathrm{Res}}_{s=1}\frac{1}{(1+(s-1))\zeta(2s)}\left(\frac{1}{s-1}+\sum_{m=0}^\infty\frac{(-1)^m}{m!}\gamma_m(s-1)^m\right)^3(xe^{(s-1)\log x})\\
         &=\mathop{\mathrm{Res}}_{s=1}\frac{x}{\zeta(2s)}\left(\frac{1}{(s-1)^3}+\frac{3\gamma}{(s-1)^2}+\frac{(3\gamma^2-3\gamma_1)}{s-1}+(\text{higher degree terms})\right)\times\\
         &\ \ \left(1+(s-1)\log x+\frac{((s-1)\log x)^2 }{2!}+\cdots\right)\left(1-(s-1)+(s-1)^2-\cdots\right)\\
         &=\mathop{\mathrm{Res}}_{s=1}\frac{1}{\zeta(2s)}\Bigg\{\frac{x}{(s-1)^3}+\frac{x(\log x+3\gamma-1)}{(s-1)^2}\\
         &+\frac{x((\log x)^2+(6\gamma-2)\log x+6\gamma^2-6\gamma-6\gamma_1+2)}{2(s-1)}+\text{higher degree terms}\Bigg\}\\
         &=\frac{1}{2\zeta(2)}x((\log x)^2+(6\gamma-2)\log x+6\gamma^2-6\gamma-6\gamma_1+2)\\
         &:=\mathcal{A}_1x(\log x)^2+\mathcal{A}_2 x\log x+\mathcal{A}_3x,
     \end{align*}
     for some real constants $\mathcal{A}_1$, $\mathcal{A}_2$ and $\mathcal{A}_3$.
     
     Now, for any simple zero $\frac{1}{4}+i\gamma_{\frac{1}{4}}$ of $\zeta(2s)$ on the line $\Re(s)=\frac{1}{4}$, we see that $\zeta^3(s)$ is analytic at $s=\frac{1}{4}+i\gamma_\frac{1}{4}$ and hence $F(s)\frac{x^s}{s}=\frac{\zeta^3(s)}{\zeta(2s)}\frac{x^s}{s}$ has a simple pole at $\frac{1}{4}+i\gamma_\frac{1}{4}$. Therefore,
\begin{align*}
    \mathop{\mathrm{Res}}_{s=\frac{1}{4}+i\gamma_\frac{1}{4}} F(s)\frac{x^s}{s}&=\zeta^3\left(\frac{1}{4}+i\gamma_\frac{1}{4}\right)\frac{x^{\frac{1}{4}+i\gamma_\frac{1}{4}}}{\frac{1}{4}+i\gamma_\frac{1}{4}}\mathop{\mathrm{Res}}_{s=\frac{1}{4}+i\gamma_\frac{1}{4}}\ \frac{1}{\zeta(2s)}\\
    &:=\mathcal{A}_{\gamma_\frac{1}{4}}x^{\frac{1}{4}+i\gamma_\frac{1}{4}}, 
\end{align*}
where $\mathcal{A}_{\gamma_\frac{1}{4}}$ is some complex constant which can be evaluated explicitly.

Note that $\zeta(0)=-\frac{1}{2}\neq0$ and therefore $F(s)=\frac{\zeta^3(s)}{\zeta(2s)}$ is analytic at $s=0$. Hence,
\begin{equation*}
    \mathop{\mathrm{Res}}_{s=0} F(s)\frac{x^s}{s}=\lim_{s\rightarrow0}(s-0)\frac{\zeta^3(s)}{\zeta(2s)}\frac{x^s}{s}=\zeta^2(0).
\end{equation*}
\end{proof}

\begin{lemma}[\cite{Titchmarsh}]
The Riemann hypothesis implies that 
    \begin{equation*}
    \zeta(\sigma+it)=O(( |t|+10)^\varepsilon),
\end{equation*}
for $\frac{1}{2}\leq\sigma\leq 2$ and $|t|\geq 10$; and 
 \begin{equation*}
    \frac{1}{\zeta(\sigma+it)}=O( ( |t|+10)^\varepsilon)
\end{equation*}
for $\frac{1}{2}<\sigma\leq 2$ and $|t|\geq 10$.\label{L3}
\end{lemma}

\section{Proof of Theorem \ref{T1}}
\begin{proof}
    
We assume the strong Riemann hypothesis for $\zeta(s)$ and $\zeta(2s)$ throughout the proof. 
By applying Perron's formula to $F(s)$, by Lemma \ref{L1}, we get
\begin{equation*}
    \sum_{n\leq x}d(n^2)=\frac{1}{2\pi i}\int_{1+\varepsilon-iT}^{1+\varepsilon+iT} F(s)\frac{x^s}{s}ds+O\left(\frac{x^{1+\varepsilon}}{T}\right),
\end{equation*}
where $10\leq T\leq x$ is a parameter to be chosen later. 

Now, we move the line of integration to $\Re(s)=-1+\varepsilon$. Then in the rectangle $\mathcal{R}$ formed by the vertices $1+\varepsilon+iT,\ -1+\varepsilon+iT,\ -1+\varepsilon-iT,\ 1+\varepsilon-iT$ joining with straight line segments, $F(s)\frac{x^s}{s}$ possesses a pole of order 3 at the point $s=1$, a simple pole at the point $s=0$ and simple poles at each zero $\frac{1}{4}+i\gamma_\frac{1}{4}$ of $\zeta(2s)$ on the line $\Re(s)=\frac{1}{4}$. Thus, Cauchy's residue theorem implies

\begin{align*}
    \sum_{n\leq x}d(n^2)&=\mathcal{A}_1x(\log x)^2+\mathcal{A}_2 x\log x+\mathcal{A}_3x+\zeta^2(0)+\sum_{\substack{\rho=\frac{1}{4}+i\gamma_\frac{1}{4}\\\zeta(2\rho)=0,\ |\gamma_\frac{1}{4}|<T}}\mathcal{A}_{\gamma_\frac{1}{4}}x^{\frac{1}{4}+i\gamma_\frac{1}{4}}\\
    &+\frac{1}{2\pi i}\left\{\int_{-1+\varepsilon-iT}^{-1+\varepsilon+iT}+\int_{-1+\varepsilon+iT}^{1+\varepsilon+iT}+\int_{-1+\varepsilon-iT}^{1+\varepsilon-iT}\right\} F(s)\frac{x^s}{s}ds+O\left(\frac{x^{1+\varepsilon}}{T}\right)
\end{align*}

Now, we evaluate the contributions of the three integrals separately. Firstly, the left vertical line integral contribution is given by 
\begin{align*}
    I_1&:=\frac{1}{2\pi i}\int_{-1+\varepsilon-iT}^{-1+\varepsilon+iT}F(s)\frac{x^s}{s}ds\\
    &\ll \int_{-1+\varepsilon-iT}^{-1+\varepsilon+iT}\left|\frac{\zeta^3(s)}{\zeta(2s)}\frac{x^s}{s}ds\right|\\
    &\ll x^{-1+\varepsilon}+x^{-1+\varepsilon}\int_{10}^T\left|\frac{\zeta^3(-1+\varepsilon+it)}{\zeta(-2+2\varepsilon+2it)}\right| t^{-1}dt\\
    &\ll x^{-1+\varepsilon}+x^{-1+\varepsilon}\int_{10}^T\left|\frac{\chi^3(-1+\varepsilon+it)\zeta^3(1-(-1+\varepsilon+it))}{\chi(-2+2\varepsilon+2it)\zeta(1-(-2+2\varepsilon+2it))}\right| t^{-1}dt\\
    &\ll x^{-1+\varepsilon}+x^{-1+\varepsilon}\int_{10}^T t^{3(\frac{1}{2}-(-1))-(\frac{1}{2}-(-2))-1+4\varepsilon} dt\\
    &\ll x^{-1+\varepsilon} T^{2+4\varepsilon},
\end{align*}
which follows from \eqref{E1.1}, \eqref{E1.2} and Lemma \ref{L3}.

Now, the horizontal line integrals contributions are given by
\begin{align*}
    I_2+I_3&:=\frac{1}{2\pi i}\left\{\int_{-1+\varepsilon+iT}^{1+\varepsilon+iT}+\int_{-1+\varepsilon-iT}^{1+\varepsilon-iT}\right\} F(s)\frac{x^s}{s}ds\\
    &\ll \int_{-1+\varepsilon}^{1+\varepsilon}\left|\frac{\zeta^3(\sigma+iT)}{\zeta(2\sigma+2iT)}\right|x^\sigma T^{-1}d\sigma\\
    &\ll \left\{\int_{-1+\varepsilon}^{\frac{1}{4}}+\int_{\frac{1}{4}}^{\frac{1}{2}}+\int_{\frac{1}{2}}^{1+\varepsilon}\right\}\left|\frac{\zeta^3(\sigma+iT)}{\zeta(2\sigma+2iT)}\right|x^\sigma T^{-1}d\sigma\\
    &\ll \int_{-1+\varepsilon}^{\frac{1}{4}}T^{3(\frac{1}{2}-\sigma)-(\frac{1}{2}-2\sigma)-1}\left|\frac{\zeta^3(1-(\sigma+iT))}{\zeta(1-(2\sigma+2iT))}\right|x^\sigma d\sigma\\
    & +\int_{\frac{1}{4}}^{\frac{1}{2}}T^{3(\frac{1}{2}-\sigma)-1}\left|\frac{\zeta^3(1-(\sigma+iT))}{\zeta(2\sigma+2iT)}\right|x^\sigma d\sigma\\
    &+\int_{\frac{1}{2}}^{1+\varepsilon}\left|\frac{\zeta^3(\sigma+iT)}{\zeta(2\sigma+2iT)}\right|x^\sigma T^{-1}d\sigma\\
    &\ll \int_{-1+\varepsilon}^{\frac{1}{4}}x^\sigma T^{-\sigma+4\varepsilon} d\sigma+\int_{\frac{1}{4}}^{\frac{1}{2}} x^\sigma T^{\frac{1}{2}-3\sigma+4\varepsilon} d\sigma+\int_{\frac{1}{2}}^{1+\varepsilon} x^\sigma T^{-1+4\varepsilon} d\sigma .
\end{align*}
which follows from \eqref{E1.1}, \eqref{E1.2} and  Lemma \ref{L3}.

Note that $(\frac{x}{T})^\sigma$ increases monotonically as a function of $\sigma$ in $[-1+\varepsilon, \frac{1}{4}]$ since $10\leq T\leq x$. However, in the other interval $[\frac{1}{4}, \frac{1}{2}]$, $(\frac{x}{T^3})^\sigma$ is only monotonic as a function of $\sigma$. Thus, we have
\begin{align*}
    I_2+I_3 &\ll x^\frac{1}{4}T^{-\frac{1}{4}+4\varepsilon}+T^{\frac{1}{2}+4\varepsilon}\left\{\left(\frac{x}{T^3}\right)^\frac{1}{2}+\left(\frac{x}{T^3}\right)^\frac{1}{4}\right\}+x^{1+\varepsilon}T^{-1+4\varepsilon}\\
    &\ll x^\frac{1}{4}T^{-\frac{1}{4}+4\varepsilon}+x^{\frac{1}{2}}T^{-1+4\varepsilon}+x^{1+\varepsilon}T^{-1+4\varepsilon}\\
    &\ll x^{1+\varepsilon}T^{-1+4\varepsilon},
\end{align*}

Therefore, in total, we have 
\begin{align*}
    \sum_{n\leq x}d(n^2)&=\mathcal{A}_1x(\log x)^2+\mathcal{A}_2 x\log x+\mathcal{A}_3x+\zeta^2(0)\\
    &+\sum_{\substack{\rho=\frac{1}{4}+i\gamma_\frac{1}{4}\\\zeta(2\rho)=0,\ |\gamma_\frac{1}{4}|<T}}\mathcal{A}_{\gamma_\frac{1}{4}}x^{\frac{1}{4}+i\gamma_\frac{1}{4}}+O\left(x^{-1+\varepsilon} T^{2+4\varepsilon}\right)+O\left(x^{1+\varepsilon}T^{-1+4\varepsilon}\right).
\end{align*}

Finally, making an optimal choice of $T$ as $\frac{T^2}{x}\sim \frac{x}{T}$, i.e., $T\sim x^{\frac{2}{3}}$, we obtain 
\begin{align*}
    \sum_{n\leq x}d(n^2)=&\mathcal{A}_1x(\log x)^2+\mathcal{A}_2 x\log x+\mathcal{A}_3x\\
    &+\sum_{\substack{\rho=\frac{1}{4}+i\gamma_\frac{1}{4}\\\zeta(2\rho)=0,\ |\gamma_\frac{1}{4}|<x^{\frac{2}{3}}}}\mathcal{A}_{\gamma_\frac{1}{4}}x^{\frac{1}{4}+i\gamma_\frac{1}{4}}+O\left(x^{\frac{1}{3}+10\varepsilon}\right).
\end{align*}

This completes the proof of Theorem \ref{T1}.
\end{proof}

\noindent {\bf Acknowledgments.} The first author wishes to express his gratitude to the Funding Agency "Ministry of Education, Govt. of India" for the Prime Minister's Research Fellowship (PMRF), ID: 3701831, for its financial support.

Addresses: K. Venkatasubbareddy, School of Mathematics and Statistics, University of Hyderabad, Hyderabad, India-500046\\
A. Sankaranarayanan, School of Mathematics and Statistics, University of Hyderabad, Hyderabad, India-500046

\end{document}